\documentclass[10pt]{article}
\usepackage{a4wide}
\usepackage{amssymb}
\usepackage{amsfonts}
\usepackage{amsmath}

\input xy
\xyoption{all}
\newtheorem{theorem}{Theorem}

\newtheorem{lemma}[theorem]{Lemma}

\newtheorem{proposition}[theorem]{Proposition}

\newenvironment{proof}[1][Proof]{\noindent\textbf{#1.} }{\ \rule{0.5em}{0.5em}}
\input{tcilatex}
\begin{document}

\date{}
\author{Halise Melis Ak\c{c}in and Ali Erdo\u{g}an }
\title{Some Results on Betti Series of Universal Modules of Differential
Operators}
\maketitle

\begin{abstract}
In this article, we discuss the rationality of the Betti series of $\Omega
_{n}(R_{m})$ where $\Omega _{n}(R_{m})$ denotes the universal module of $n$%
th order derivations of $R_{m}$. We proved that if $R$ is a coordinate ring
of an affine irreducible curve represented by $\frac{k[x_{1},x_{2},...,x_{s}]%
}{(f)}$ and if it has at most one singularity point, then the Betti series of $%
\Omega _{n}(R_{m})$ is rational where $m$ is a maximal ideal of $R$.

$\noindent $
\end{abstract}


\let\thefootnote\relax\footnote{\emph{Key Words: universal differential operator
modules, minimal resolution }

\emph{Mathematics subject classification 2010: 13N05}}


\section{Introduction and Preliminaries}

The following notations will be fixed throughout the paper: ring means
commutative with identity and $R$ is a commutative $k$-algebra where $k$ is
an algebraically closed field of characteristic zero. \newline

An $n$th order $k$-derivation $D$ of $R$ into an $R$-module $F$ is an
element of $Hom_{k}(R,F)$ such that for any $n+1$ elements $r_0, r_1,\ldots,
r_{n}$ of $R$, the following identity holds:

\begin{center}
$D(r_0r_1\ldots r_{n})=\sum\limits_{i=1}^{n}(-1)^{i-1}
\sum\limits_{j_1<j_2<\ldots<j_{i}}r_{j_1}\ldots r_{j_{i}}D(r_0\ldots \hat{r}%
_{j_1}\ldots \hat{r}_{j_{i}}\ldots r_{n})$
\end{center}
where the hat over $r_{i}$'s means that it is missed. It can be
easily seen that a first order derivation is the ordinary
derivation of $R$ into an $R$-module $F$.

In \cite{Nak2}, a universal object for $n$th order derivations
constructed in the following way: Consider the exact sequence
\begin{equation*}
0\rightarrow I\rightarrow R\otimes _{k}R\text{ }\overset{\varphi }{%
\rightarrow }\text{ }R\rightarrow 0
\end{equation*}%
where $\varphi $ is defined as $\varphi
(\sum\limits_{i=1}^{n}a_{i}\otimes
b_{i})=\sum\limits_{i=1}^{n}a_{i}b_{i}$ for $a_{i},b_{i}\in R$ and
$I$ \ is the kernel of $\varphi$ . It is known that $\ker \varphi$
is generated by the set

\begin{center}
\{$1\otimes r-r\otimes 1:r\in R$\}
\end{center}
as an $R$-module. Then the mapping $d_{n}$ from $R$ into
$I/I^{n+1}$ given by

\begin{equation*}
d_{n}(r)=1\otimes r-r\otimes 1+I^{n+1} \text{ and } d_{n}(1)=0
\end{equation*}%
is called the universal derivation of order $n$ that is, any $n$th order
derivation $D$ from $R$ into $F$ can be factored through $I/I^{n+1}$ where $F
$ is an $R$-module. Here, the $R$-module $I/I^{n+1}$ is called the universal
module of $n$th order derivations and is denoted by $\Omega _{n}(R)$.

Note that if $R$ is a finitely generated $k$-algebra, then $\Omega _{n}(R)$
is a finitely generated $R$-module. It is proved in \cite[Prop. 2]{Nak2}
that if $R=k[x_{1},...,x_{s}]$ is a polynomial algebra over $k$ with $s$
variables, then $\Omega _{n}(R)$ is a free $R$-module of rank $\binom{n+s}{s}%
-1$ with basis
\begin{equation*}
\left\{ d_{n}(x_{1}^{\alpha _{1}}x_{2}^{\alpha _{2}}...x_{s}^{\alpha _{s}}):%
\text{ }1\leq \alpha _{1}+\alpha _{2}+...+\alpha _{s}\leq n\right\}
\end{equation*}%
and in \cite[Theo. 9]{Nak2} that

\begin{center}
$\Omega _{n}(R)\otimes _{R}R_{S}\cong \Omega _{n}(R_{S})$
\end{center}
where $S$ is a multiplicatively closed subset of $R$.

A free resolution of $\Omega _{n}(R)$ where $R$ is a local $k$-algebra with
maximal ideal $m$ is called a minimal resolution if the followings are
satisfied:

\begin{equation*}
\ldots \rightarrow F_{2}\text{ }\overset{\partial _{2}}{\rightarrow }\text{ }%
F_{1}\overset{\partial _{1}}{\rightarrow }F_{0}\text{ }\overset{\varepsilon }%
{\rightarrow }\text{ }\Omega _{n}(R)\rightarrow 0
\end{equation*}%
$F_{i}$'s are free $R$-modules of finite rank for all $i$ and $\partial
_{n}(F_{n})\subseteq mF_{n-1}$ for all $n\geq 1$ (see \cite{Ku} for
definition).

Let $(R,m)$ be a local ring. The Betti series of $\Omega_{n}(R)$ is defined
to be the series

\begin{equation*}
B(\Omega _{n}(R),t)=\underset{i\geq 0}{\sum }\dim _{R/m}Ext^{i}(\Omega
_{n}(R),\frac{R}{m})t^{i}\text{ for all }n\geq 1.
\end{equation*}

\begin{lemma}
Let $R$ be a local ring with maximal ideal $m$ and $M$ be a finitely
generated $R$-module. Suppose that
\begin{equation*}
0\rightarrow F_{1}\overset{\partial }{\rightarrow }F_{0}\rightarrow
M\rightarrow 0
\end{equation*}%
is a minimal resolution of $M.$ Then $Ext^{1}(M,R/m)$ is not zero.
\end{lemma}

Dealing with $n$th order case presents extra difficulties. Let
$R=k[x,y,z]$ be a $k$-algebra with $z^2=x^3$ and $y^2=xz$. It is
shown in \cite[ex. 3.1.6 and 3.4.7]{Er3} that
$pd(\Omega_1(R))\leq1$, but $%
pd(\Omega_2(R))$ is not finite.

\bigskip In \cite{Er4}, the following proposition is proved:

\begin{proposition}
Let $R=k[x_{1}, \ldots, x_{s}]$ and $S=k[y_{1}, \ldots, y_{t}]$ be
polynomial algebras and let $I$ be an ideal of $R$ generated by elements $%
\{f_{1}, \ldots, f_{m}\}$. Assume that $R/I=k[x_{1}, \ldots, x_{s}]/ (f_{1},
\ldots, f_{m})$ is an affine $k$-algebra of dimension $s-m$ and $%
pd(\Omega_{2}(R/I))\leq 1$. Then
$pd(\Omega_{2}(R/I\otimes_{k}S))\leq 1$.
\end{proposition}
But, unfortunately, this result is not true even for $n=3$. \ \ \
So, there are two natural questions arise from these results. Can
we generalize the dimension of the ring $R$? Can we generalize the
dimension of the universal module $\Omega_{n}$? In \cite{Er1} and
\cite{Mel1}, it is studied the following question:

\begin{center}
When is the Betti series of a universal module of second order derivations
rational?
\end{center}
Our goal is to establish a result analogue of this question for
$n$th order universal differential operator modules.

\newpage
\section{Main Results}

\begin{proposition}
Let $k[x_1,x_2,\ldots, x_{s}]$ be a polynomial algebra and $m$ be a maximal
ideal of $k[x_1,x_2,\ldots, x_{s}]$ containing an irreducible element $f$.
If the elements $d_{n}(x_1^{\alpha_1}x_2^{\alpha_2} \ldots
x_{s}^{\alpha_{s}}f)$ belong to $m\Omega_{n}(k[x_1,x_2,\ldots, x_{s}]) $
whenever $0\leq \alpha_1+ \alpha_2 + \ldots + \alpha_{s} \leq n-1$, then $%
\Omega_{n}(\frac{k[x_1,x_2,\ldots, x_{s}]}{(f)})_{\overset{-}{m}}$ admits a
minimal resolution of $(\frac{k[x_1,x_2,\ldots, x_{s}]}{(f)})_{\overset{-}{m}%
}-$modules where $\overset{-}{m}=m/(f)$ is a maximal ideal of $\frac{%
k[x_1,x_2,\ldots, x_{s}]}{(f)}$.
\end{proposition}

\begin{proof}
Let $R=S/I=\frac{k[x_1,x_2,\ldots, x_{s}]}{(f)} $ and $\overset{-}{m}$ be a
maximal ideal of $R$. Then by \cite[Theo. 14 pg. 24]{Nak2} we have the
following short exact sequence of $R$-modules

\begin{equation}  \label{seq}
\xymatrix{0 \ar[r] & \frac{N+f\Omega_{n}(S)}{f\Omega_{n}(S)} \ar[r] &
\frac{\Omega_{n}(S)}{f\Omega_{n}(S)} \ar[r]^{\alpha} & \Omega_{n}(R) \ar[r]
& 0 }
\end{equation}
where $N$ is a submodule of $\Omega_{n}(S)$ generated by the elements of the
form

\begin{center}
$\{d_{n}(g): g \in fk[x_1,x_2,\ldots, x_{s}]\}$.
\end{center}
By localizing (\ref{seq}) at $\overset{-}{m}$, we get the
following exact sequence of $R_{\overset{-}{m}}-$modules:

\begin{equation}  \label{seq1}
\xymatrix{0 \ar[r] &
(\frac{N+f\Omega_{n}(S)}{f\Omega_{n}(S)})_{\overset{-}{m}} \ar[r] &
\frac{\Omega_{n}(S)}{f\Omega_{n}(S)}_{\overset{-}{m}}
\ar[r]^{\alpha_{\overset{-}{m}}} & \Omega_{n}(R)_{\overset{-}{m}} \ar[r] & 0
}.
\end{equation}

\bigskip

\textbf{Step 1.} A module generated by the set $\{d_{n}(g): g \in
f k[x_1,x_2,\ldots, x_{s}]\}$ is a submodule of
$m\Omega_{n}(k[x_1,x_2,\ldots, x_{s}])$.

\bigskip

\textit{Proof of Step 1}. Since $d_{n}$ is $k$-linear, it suffices to show
\begin{center}
$d_{n}(x_1^{\alpha_1}x_2^{\alpha_2} \ldots x_{s}^{\alpha_{s}}f)\in
m\Omega_{n}(k[x_1,x_2,\ldots, x_{s}])$.
\end{center}
By using the properties of $d_{n}$, we get

$d_{n}(x_1^{\alpha_1}x_2^{\alpha_2} \ldots x_{s}^{\alpha_{s}}f)=\underset{%
\gamma}{\Sigma}a_{\gamma}(x_1,x_2,\ldots,
x_{s})d_{n}(x_1^{\gamma_1}x_2^{\gamma_2} \ldots
x_{s}^{\gamma_{s}}f)$ $+f(\underset{\beta}{\Sigma}a^{\prime
}_{\beta}(x_1,x_2,\ldots, x_{s})d_{n}(x_1^{\beta_1}x_2^{\beta_2}
\ldots x_{s}^{\beta_{s}})$ where $a_{\gamma}(x_1,x_2,\ldots,
x_{s}), a^{\prime }_{\beta}(x_1,x_2,\ldots, x_{s})\in
k[x_1,x_2,\ldots, x_{s}] $, $0\leq \gamma_1+\gamma_2+\ldots
+\gamma_{s}\leq n-1$, $0< \beta_1+ \beta_2+\ldots + \beta_{s}\leq
n $.
By assumption, we know
\begin{center}
$d_{n}(x_1^{\gamma_1}x_2^{\gamma_2} \ldots x_{s}^{\gamma_{s}}f)
\in m\Omega_{n}(k[x_1,x_2,\ldots, x_{s}])$
\end{center}
whenever $0\leq \gamma_1+ \gamma_2 + \ldots + \gamma_{s} \leq n-1$
and $f\in m$, then the result follows.

$\noindent $

\bigskip

\textbf{Step 2.} $(\frac{N+f\Omega_{n}(S)}{f\Omega_{n}(S)})_{\overset{-}{m}%
}\subseteq \overset{-}{m}(\frac{\Omega_{n}(S)}{f\Omega_{n}(S)})_{\overset{-}{%
m}}$.

\bigskip

\textit{Proof of Step 2}. By step 1, we know $N\subseteq m\Omega_{n}(S)$ and
the rest is clear.

\bigskip

\textbf{Step 3.} $(\frac{N+f\Omega_{n}(S)}{f\Omega_{n}(S)})_{\overset{-}{m}}$
is generated by $\binom{n+s-1}{s}$ elements.

\bigskip

\textit{Proof of Step 3}. It is known that $(\frac{N+f\Omega_{n}(S)}{%
f\Omega_{n}(S)})$ is generated by the set
\begin{center}
$\{d_{n}(x_1^{\alpha_1}x_2^{\alpha_2} \ldots x_{s}^{\alpha_{s}}f)+
f\Omega_{n}(S): 0\leq \alpha_1+ \alpha_2+ \ldots + \alpha_{s}\leq
n-1\}$.
\end{center}
And, it has $\binom{n+s-1}{s}$ elements.
\bigskip

\textbf{Step 4.} $(\frac{N+f\Omega_{n}(S)}{f\Omega_{n}(S)})_{\overset{-}{m}}$
is a free $R_{\overset{-}{m}}$-module.

\bigskip

\textit{Proof of Step 4}. The Krull dimension of $R_{\overset{-}{m}}$ is $s-1
$ and let $K$ be the field of fractions of $R_{\overset{-}{m}}$. Then by
tensoring the exact sequence in (\ref{seq1}) by $K$, we get

\begin{equation}  \label{seq2}
\xymatrix{0 \ar[r] & K
\otimes_{R_{\overset{-}{m}}}(\frac{N+f\Omega_{n}(S)}{f\Omega_{n}(S)})_{%
\overset{-}{m}} \ar[r] & K \otimes_{R_{\overset{-}{m}}}
(\frac{\Omega_{n}(S)}{f\Omega_{n}(S)})_{\overset{-}{m}}
\ar[r]^{\alpha_{\overset{-}{m}}} & K \otimes_{R_{\overset{-}{m}}}
\Omega_{n}(R)_{\overset{-}{m}} \ar[r] & 0 }.
\end{equation}
We know that,
$(\frac{\Omega_{n}(S)}{f\Omega_{n}(S)})_{\overset{-}{m}}$ is a
free $R_{\overset{-}{m}}$- module of rank $\binom{n+s}{s}-1$.
By
using the following isomorphism,
\begin{center}
$K\otimes_{R_{\overset{-}{m}}}\Omega_{n}({R_{\overset{-}{m}}})\cong
\Omega_{n}(K)$
\end{center}
we have

$dim$ $K \otimes_{R_{\overset{-}{m}}}(\frac{N+f\Omega_{n}(S)}{f\Omega_{n}(S)}%
)_{\overset{-}{m}}= $dim$ $K$ \otimes_{R_{\overset{-}{m}}} (\frac{\Omega_{n}(S)}{%
f\Omega_{n}(S)})_{\overset{-}{m}}-dim \Omega_{n}(K)$
$=\binom{n+s}{s}-\binom{n+s-1}{s-1}=\binom{n+s-1}{s}$.
Since the
rank of
$(\frac{N+f\Omega_{n}(S)}{f\Omega_{n}(S)})_{\overset{-}{m}}
$ is equal to its number of minimal generators, it is a free $R_{\overset{-}{%
m}}$-module. Therefore, the short exact sequence given in (\ref{seq1}) is a
minimal resolution for $\Omega_{n}({R_{\overset{-}{m}}})$.
\end{proof}

\bigskip
Let $R$ be a finitely generated regular $k$-algebra and $m$ be a
maximal ideal of R then $\Omega_{n}({R_{m}})$ is a free
$R_{m}$-module. Then it is clear that $B(\Omega_{n}({R_{m}},t))$
is rational.

\bigskip

\begin{theorem}
Let $k[x_1,x_2,\ldots, x_{s}]$ be a polynomial algebra and $m$ be
a maximal ideal of $k[x_1,x_2,\ldots, x_{s}]$ containing an
irreducible element $f$. Let $d_{n}(x_1^{\alpha_1}x_2^{\alpha_2}
\ldots x_{s}^{\alpha_{s}}f)\in m\Omega_{n}(k[x_1,x_2,\ldots,
x_{s}])$
for $0\leq \alpha_1+ \alpha_2 + \ldots + \alpha_{s} \leq n-1$. Assume that $%
R=\frac{k[x_1,x_2,\ldots, x_{s}]}{(f)}$ is not a regular ring at $\overset{-}%
{m}=m/(f)$. Then $B(\Omega_{n}({R_{\overset{-}{m}}},t),)$ is a rational
function.
\end{theorem}

\bigskip

\begin{proof}
By the previous proposition, the exact sequence of $R_{\overset{-}{m}}$%
-modules in (\ref{seq1}) is a minimal resolution of $\Omega_{n}({R_{\overset{%
-}{m}}})$. And we get the result.
\end{proof}

\bigskip

\textbf{Question:} It should be interesting to know whether the
Betti Series is rational for the algebra $R=k[x,y,z]$ with
$z^2=x^3$ and $y^2=xz$.

\bigskip

\textbf{Acknowledgement}

This work is a part of the first autor's PhD thesis. 
The first author is grateful to TUBITAK for their financial
support during her visit at the University of Sheffield.

\bigskip
Ali Erdogan \\
Hacettepe University \\
Department of Mathematics \\
Ankara,Turkey \\
e-mail:alier@hacettepe.edu.tr

\bigskip
Halise Melis Akcin \\
Hacettepe University \\
Department of Mathematics \\
Ankara,Turkey \\
e-mail:hmtekin@hacettepe.edu.tr

\end{document}